\documentclass[12pt, reqno]{amsart}

\usepackage{amsmath, amsthm, amscd, amsfonts, amssymb, graphicx, color}
\usepackage[bookmarksnumbered, colorlinks, plainpages]{hyperref}

 \makeatletter \oddsidemargin.9375in
\evensidemargin \oddsidemargin \marginparwidth 2in \makeatother

\newtheorem{theorem}{Theorem}[section]
\newtheorem{cor}[theorem]{Corollary}
\newtheorem{lem}[theorem]{Lemma}
\newtheorem{prop}[theorem]{Proposition}
\newtheorem{defn}[theorem]{Definition}

\newtheorem{exam}[theorem]{Example}

\numberwithin{equation}{section}

\newcommand{\BR}{{\Bbb  R}}

\newcommand{\Lip}{{\rm Lip}}
\newcommand{\lip}{{\rm lip}}

\newcommand{\St}{{\rm St}}

\begin{document}

\title[Isometries on certain non-complete vector-valued function
spaces]{Isometries on certain non-complete vector-valued function
spaces}

\author{Mojtaba  Mojahedi and  Fereshteh Sady$^1$}



\maketitle
\begin{center}

\address{{\em   Department of Pure
Mathematics, Faculty of \\ Mathematical Sciences,
 Tarbiat Modares University,\\ Tehran, 14115-134, Iran}}

\vspace*{.25cm}
  \email{mojtaba.mojahedi@modares.ac.ir},
  \email{sady@modares.ac.ir}
\end{center}
\footnote{$^1$ Corresponding author}

\maketitle

\begin{abstract}
In the recent paper \cite{Hos}, surjective isometries, not necessarily linear, $T: {\rm AC}(X,E)
\longrightarrow {\rm AC}(Y,F)$ between vector-valued absolutely
continuous functions on compact subsets $X$ and $Y$ of the real
line, has been described. The target
spaces $E$ and $F$ are strictly convex  normed spaces. In this paper, we assume that $X$ and $Y$ are compact Hausdorff spaces and $E$ and $F$ are normed spaces, which are not assumed to be strictly convex. We describe (with a short proof)  surjective isometries $T: (A,\|\cdot\|_A) \longrightarrow (B,\|\cdot\|_B)$ between certain normed subspaces $A$ and $B$ of $C(X,E)$ and $C(Y,F)$, respectively.  We consider three cases for $F$ with some mild conditions.
The first case, in particular,  provides a short proof for the above result, without
assuming that the target spaces are strictly convex. The other cases give some generalizations in this topic.
 %
 %
 As a consequence, the  results can be applied, for  isometries (not necessarily linear) between  spaces of  absolutely continuous vector-valued functions, (little) Lipschitz functions and also continuously differentiable functions.
\end{abstract}

\section{introduction}
The characterization of linear isometries between specific Banach spaces of continuous functions is a longstanding problem. By the classical Banach-Stone theorem any linear isometry $T$ from the space $C(X)$ of continuous functions on a compact Hausdorff space $X$ onto  $C(Y)$, where $Y$ is a compact Hausdorff space, is a weighted composition operator.
 The theorem has various generalizations in different settings with respect to certain norms. The first result concerning the vector-valued generalization of Banach-Stone theorem has been given by Jerison in \cite{Jer}. By Jerison's result, if $E$ is a strictly convex Banach space, then any surjective linear isometry $T: C(X,E) \longrightarrow C(Y,E)$ between the spaces of all  $E$-valued continuous functions on compact Hausdorff spaces  $X$ and $Y$ has a canonical form.
This result has been  generalized  in different directions by considering surjective linear isometries between certain subspaces of continuous functions equipped with either supremum norm or  certain  complete norms. For the study of some known results in this topic we refer the reader to the nice books \cite{Flem1,Flem2}.

For a metric space $(X,d)$ and a normed space $E$, let ${\rm Lip}(X,E)$ be the space of all bounded $E$-valued Lipschitz functions $f$ on $X$ and let $L(f)$ be the Lipschitz constant of $f$, i.e. \[L(f)= \sup \{ \frac{\|f(x) - f(y)\|}{d(x,y)}: x,y
\in X, x\neq y\}.\] Then ${\rm Lip(X,E)}$ is a normed space under the norm $\|f\|=\max(\|f\|_\infty, L(f))$, $f\in {\rm Lip}(X,E)$, where $\|\cdot\|_\infty$ denotes the supremum norm.
Surjective linear isometries between spaces of Lipschitz functions (on compact metric spaces) with values in strictly convex Banach  spaces have been studied  in  \cite{Var}.  In \cite{Ar}, Araujo and Dubarbie study surjective linear isometries $T: {\rm Lip}(X,E) \longrightarrow {\rm Lip}(Y,F)$ between Lipschitz spaces of functions on metric spaces $X$ and $Y$,  not necessarily compact, with values in strictly convex (not necessarily complete) normed spaces $E$ and $F$. They assume that the  given isometry $T$ satisfies a condition called property {\bf P} and  investigate for the conditions under which $T$ is a weighted composition operator. By property {\bf P} for each $y\in Y$ there exists a constant function $f\in {\rm Lip}(X,E)$ with $Tf(y)\neq 0$. In \cite{Bot}, Botelho, Fleming and Jamison, by removing the strict convexity assumption, study surjective linear isometries $T: {\rm Lip}(X,E) \longrightarrow {\rm Lip}(Y,F)$  for the case that $X$ and $Y$ are compact metric spaces and $E$ and $F$ are quasi-sub-reflexive Banach spaces with trivial centeralizers. Then the result of \cite{Bot} has been improved in \cite{Ran} by removing quasi-sub-reflexivity assumption. We should note that in both \cite{Bot} and \cite{Ran},  $T$ is assumed to satisfy a property called {\bf Q} which is stronger than property {\bf P}.  This property asserts that for each $y\in Y$ and $u\in F$ there exists a constant function $f\in {\rm Lip}(X,E)$ such that $Tf(y)=u$.

More recently, in \cite{Hos} M. Hosseini gives a characterization of surjective, not necessarily linear, isometries $T:{\rm AC}(X,E) \longrightarrow {\rm AC}(Y,F)$ between spaces of $E$-valued and $F$-valued absolutely continuous functions on compact subsets $X$ and $Y$ of the real line, under the requirement that $E$ and $F$ are strictly convex normed spaces. By the Mazur-Ulam theorem such isometries are real-linear up to a translation. We should note that, as in \cite{Ar}, the approach of  \cite{Hos} is based on the relationship between such isometries and separating maps and the proofs do not work without  strict convexity assumption on the target spaces.

   In this paper we assume that $X$ and $Y$ are compact Hausdorff spaces and $E$ and $F$ are normed spaces which are not assumed to be strictly convex. We describe (with a short proof) surjective , not necessarily linear, isometries  $T: (A, \|\cdot\|_A) \longrightarrow (B,\|\cdot\|_B)$ between certain subspaces $A$ and $B$ of  $C(X,E)$ and $C(Y,F)$, respectively, equipped with some norms $\|\cdot\|_A$ and $\|\cdot\|_B$.  Our results will be stated for three cases for $F$. The hypothesis in each case is fulfilled by any strictly convex normed space.   Meanwhile, vector-valued function spaces such as spaces of vector-valued Lipschitz functions, absolutely continuous functions and  continuously differentiable functions satisfy our requirements. In particular, Case (i) in our results (Theorem \ref{main2}) provides a short proof for the main result of \cite{Hos}  without assuming that $E$ and $F$ are strictly convex.

\section{preliminaries}
Let $\Bbb K$ be the field of real or complex numbers. For a normed
space $E$ over $\Bbb K$ we denote its closed unit ball by $E_1$
and its unit sphere by $S(E)$. The notations $E^*$ and ${\rm ext}(E^*_1)$
will be used for the dual space of $E$ and the set of extreme points of $E^*_1$, respectively. For each
$e\in S(E)$ the star-like set ${\rm St}(e)$ is defined as
\[ \St(e)=\{ e'\in S(E): \|e+e'\|=2 \}.\] Then $\St(e)$  is the union
of all maximal convex subsets of $S(E)$ containing $e$. If $E$ is
strictly convex, then it is obvious that $\St(e)=\{e\}$ for all
$e\in E$. It is easy to see that if $u,v\in E$ satisfy $\|u+v\|=\|u\|+\|v\|$, then for all $\lambda, \mu \ge 0$ we have
$\|\lambda u+\mu v\|=\lambda \|u\|+\mu \|v\|$, see \cite[Lemma 4.1] {Jer}. Hence, if  $e\in S(E)$, then for each $e'\in \St(e)$ and $r>0$ we have $\|re+e'\|=r+1$. In particular, $\|re+e'\|>r=\|re\|$. Motivated by this, for each $u\in E$ we put
\[ \St_w(u)=\{e'\in S(E) : \|u+e'\|>\|u\|\}.\]
We note that if $u,v\in E$ and $\|u+v\|>\|u\|$, then $\|u+rv\|>\|u\|$ for all $r\in [1,+\infty)$. Hence
for the given $u\in E$ if $v\in E$ with $\|v\|\le 1$ satisfies $\|u+v\|>\|u\|$, then
$\|u+\frac{v}{\|v\|}\|>\|u\|$, that is  $\frac{v}{\|v\|}\in \St_w(u)$.

For a topological space $X$ and a normed space $E$ over $\Bbb K$, let $C(X,E)$ be the space of all continuous $E$-valued functions on $X$. For  an element $v$ in a normed space
$E$, the constant function $x \mapsto v$ in $C(X,E)$ will be denoted by $\hat{v}$.
For the case where $X$ is locally compact, $C_0(X,E)$ denotes the normed space of all continuous $E$-valued
functions on $X$ vanishing at infinity, with the
supremum norm $\|\cdot\|_\infty$.  Furthermore,  $C(X)$ and $C_{\Bbb
R}(X)$ denote the Banach spaces of complex-valued, respectively
real-valued continuous functions on a compact Hausdorff space $X$.

 Let  $\mathcal{Z}=C_0(X,E)$ where $X$ is a locally compact Hausdorff space and $E$ is a (real or complex) normed space.
Then, by \cite[Theorem 2.3.5]{Flem1}, we have \[ {\rm ext}(\mathcal{Z}^*_1)=\{v^*\circ \delta_x: v^*\in {\rm ext}(E^*_1), x\in X\},\]
where for each $x\in X$, $\delta_x: C_0(X,E)\longrightarrow E$ is defined by $\delta_x(f)=f(x)$.

A T-set in a normed space $E$ (over $\Bbb K$) is a subset $S$ of $E$ which is
maximal with respect to the property that  $\|e_1+\cdots
e_n\|=\|e_1\|+\cdots +\|e_n\|$ holds for any $n\in \Bbb N$ and $e_1,...,e_n\in S$.
A normed space $E$ is said to satisfy  property
(D) if for any pair of distinct  T-sets $S_1$ and $S_2$ we have
either $S_1\cap S_2=\{0\}$ or there exists a T-set $L$ with
$S_1\cap L=\{0\}=S_2\cap L$ (see Definition 7.2.10 in \cite{Flem2}). It is obvious that any strictly convex normed space satisfies property (D).
For some examples of nonstrictly convex normed spaces with this property, see Examples 7.2.11 in \cite{Flem2}.

Let $X$ be a  locally compact Hausdorff space and $E$ be a normed space. For each T-set $S$ in $E$ and any $t\in X$, the set
\[ (S,t)=\{ f\in C_0(X,E): f(t)\in S, \; {\rm and }\; \|f(t)\|=\|f\|_\infty \}\]
is a T-set in $C_0(X,E)$, and conversely any T-set in $C_0(X,E)$ is of this form, see  \cite[Lemma 7.2.2]{Flem2}. We should note that the lemma has been proven for the case that $E$ is a Banach space, however, the completeness of $E$ has no role in the proof.

Let $X$ be a compact Hausdorff space.
For a subspace $A$ of $C(X,\Bbb K)$ and a  normed space $E$ over
$\Bbb K$, let $A\otimes E$ be the linear span of $\{fe:  f\in A,
e\in E\}$ where for $f\in A$ and $e\in E$, the $E$-valued
continuous function $fe$ is defined by $fe(x)=f(x) e$, $x\in X$.
It is well-known that if $X$ is compact and $E$ is a complex
Banach algebra, then $C(X)\otimes E$ is dense in $C(X,E)$ , see \cite[Lemma 1]{Haus}. However
the same proof works for any normed space $E$. For the sake of
completeness we state it here.

\begin{lem}
Let $X$ be a Hausdorff space and $E$ be a normed space over $\Bbb
K$. Then for any $f\in C(X,E)$ with compact range,  and each
$\epsilon>0$ there exist $f_1, f_2, ...,f_n\in C_{\Bbb R}(X)$ with
compact ranges and elements $e_1,...,e_n\in E$ such that
$\sup_{x\in X}\|\Sigma_{i=1}^n f_i(x) e_i-f(x)\|\le \epsilon$.
\end{lem}
\begin{proof}
For each $x\in X$, we put $V_x=\{e\in E: \|e-f(x)\|<\epsilon\}$.
Then clearly $f(X) \subseteq \cup_{x\in X}V_x$. Hence by
compactness assumption, there are $x_1,x_2,...,x_n\in X$ such that
$f(X) \subseteq \cup_{i=1}^n V_{x_i}$. Using partition of unity,
there are $\lambda_1,...,\lambda_n \in C_{\Bbb R}(f(X))$ such that
$\Sigma_i{\lambda_i}=1$ on $f(X)$ and ${\rm supp(\lambda_i})
\subseteq V_{x_i}$. For each $i=1,..,n$, put $e_i=f(x_i)$ and let
$f_i\in C_{\Bbb R} (X)$ be defined by $f_i=\lambda_i \circ f$. Then it is
easy to see that for each  $x\in X$, we have $\|\Sigma_{i=1}^n
f_i(x) e_i-f(x)\|\le \epsilon$, as desired.
\end{proof}

\begin{cor}
Let $X$ be a compact Hausdorff space and let $E$ be a normed space
over $\Bbb K$.  Then $C(X,\Bbb K)\otimes E$ is dense in $C(X,E)$.
\end{cor}
\begin{proof}
This is immediate from the above lemma.
\end{proof}

By the above corollary, any subspace $A(X,E)$ of $C(X,E)$
containing $A\otimes E$ for some dense subspace $A$ of $C(X,\Bbb
K)$, is dense in $C(X,E)$. In particular, the  $E$-valued
function spaces on $X$  in the next examples are all dense in the $C(X,E)$.

\begin{exam}\label{exam}
{\rm Let $E$ be a normed space over $\Bbb K$.

(i) For a compact metric space  $(X,d)$,  and $\alpha \in (0,1]$, we denote the space of all
$E$-valued Lipschitz functions of order $\alpha$ on $X$ by
$\Lip^\alpha(X , E)$. Then $\Lip^\alpha(X,E)$ is a normed space
with respect to the norm $\|f\|_\alpha=\max( \|f\|_\infty, L(f))$,
where \[L(f)= \sup \{ \frac{\|f(x) - f(y)\|}{d^\alpha(x,y)}: x,y
\in X, x\neq y\}.\] Meanwhile it is complete whenever $E$ is a Banach
space.

For $\alpha\in (0,1)$, the closed subspace $\lip^\alpha(X , E)$ of $\Lip^\alpha(X , E)$
consists of those functions $f \in \Lip^\alpha(X , E)$  such that
$ \lim_{d(x,y) \to 0} \frac{\|f(x) - f(y)\|}{d^\alpha(x,y)}=0.$ We
write $\Lip^\alpha(X)$ and $\lip^\alpha(X)$ for complex-valued
case.

(ii) For a compact subset $X$ of the real line, let ${\rm AC}(X,E)$ be the space of all absolutely continuous $E$-valued functions on  $X$. Then
\[ \|f\|=\max( \|f\|_\infty ,{\rm var}(f)),\;\;\;\; (f\in {\rm AC}(X,E))\]
defines a norm on ${\rm AC}(X,E)$, where ${\rm var}(f)$ is the total variation of $f\in {\rm AC}(X,E)$.

(iii) For $n\in \Bbb N$, the space $C^n([0,1],E)$ consists of all  $n$-times continuously
 differentiable $E$-valued functions on $X=[0,1]$.  Meanwhile,
\[ \|f\|=\max\{\|f^{(i)}\|_\infty, i=0,...,n\} \;\;\;\; (f \in C^n([0,1],E)),\]
is a norm on $C^n([0,1],E)$.

(iv) For $n\in \Bbb N$,  ${\rm Lip}^n([0,1],E)$ denotes the space of all $n$-times differentiable
$E$-valued Lipschitz functions  on $[0,1]$ whose derivatives are also
Lipschitz functions. Then
\[ \|f\| =\max\{\|f^{(i)}\|_\infty, L(f^{(i)}), i=0,...,n\}\]
defines a norm on ${\rm Lip}^n([0,1],E)$.
}
\end{exam}
\section{Main Results}
Our main theorem (Theorem \ref{main2}) characterizes surjective isometries $T$ between
certain (not necessarily complete) normed spaces of vector-valued
function spaces.  First we prove the result for the case that the domain and the range of $T$ is the whole space of continuous functions.

It should be noted that Case (i) in Theorem \ref{main2} provides a
short proof for the main result of \cite{Hos} without assuming
that the target spaces are strictly convex. Cases (ii) and (iii)
give some other generalizations in this topic. For example of a
non-strictly convex normed space $F$ satisfying the hypothesis of
Case (i) see \cite{Jam}. Indeed, there are infinitely many norms
$\|\cdot\|$ on $\BR^2$ such that $(\BR^2,\|\cdot \|)$ is
non-strictly convex and satisfies the hypothesis of Case(i). 

Since, by the Mazur-Ulam theorem, any  surjective isometry between
real normed spaces is real-linear up to a translation, we assume
that the given isometries are  real-linear.
\begin{theorem} \label{main1}
Let $X$ and $Y$ be  locally compact Hausdorff spaces and let $E$
and $F$ be  real or complex normed (not necessarily complete)
spaces. Let $T: C_0(X,E) \longrightarrow C_0(Y,F)$ be a surjective
real-linear isometry. Assume that one of the following conditions
holds:

{\rm (i)} $S(F)$ contains an element $v_0$ with ${\rm
St}(v_0)=\{v_0\}$.

{\rm (ii)} ${\rm ext}(F_1^*)$ contains an element $w_0^*$ with
${\rm St}(w_0^*)=\{w_0^*\}$.

{\rm(iii)} $F$ satisfies property {\rm (D)}.

\noindent Then there exist a continuous map  $\Phi:Y \longrightarrow X$, a family $\{V_y\}_{y\in Y}$ of bounded real-linear operators from $E$ to $F$ with $\|V_y\|\le 1$ such that for each $y\in Y$
\[Tf(y)= V_y(f(\Phi(y))) \qquad (f\in C_0(X,E)).\]
Moreover, if $E$ satisfies the same condition as $F$, then $\Phi$ is a homeomorphism and each $V_y$ is a surjective  isometry.
\end{theorem}
\begin{proof}
Without loss of generality we assume that $E$ and $F$ are real normed spaces. Let $C_0(X,E)^*$ and $C_0(Y,F)^*$ denote the (real) duals of $C_0(X,E)$ and $C_0(Y,F)$, respectively, and  $T^*: C_0(Y,F)^*
\longrightarrow C_0(X,E)^*$ be the adjoint of $T$ as a bounded real-linear
 operator. Then $T^*$ is a surjective real-linear isometry and hence for each pair $(w^*,y) \in {\rm ext}(F^*_1) \times Y$ there
exists a pair $(v^*,x) \in {\rm ext}(E^*_1) \times X$ such that
\begin{align*}
 w^*(Tf(y))=v^*(f(x))\;\;\;\;\; (f\in C_0(X,E)).
\end{align*}

{\bf Case (i).} By assumption, the singleton $\{v_0\}$
is a maximal convex subset of $S(F)$. Choose $w^* \in {\rm ext}
(F^*_1)$ such that $w^*(v_0)=\|v_0\|=1$. Then the  set $R=\{u\in
S(F) : w^*(u)=1\}$  is a convex subset of $S(F)$ containing  $v_0$
and since $\{v_0\}$ is maximal we get $R=\{v_0\}$. Fix a point
$y\in Y$. Since $w^*\circ \delta_y$ is an extreme point of the unit ball of $C_0(Y,F)^*$   there
are $v^*\in {\rm ext}(E^*_1)$ and $x\in X$ such that
  \begin{align}\label{ext}
  w^*(Tf(y))=v^*(f(x))\;\;\;\;\; (f\in C_0(X,E)).
  \end{align}
We note that if $v_1^*\in {\rm ext}(E^*_1)$ and $x_1\in X$ such that $v^*(f(x))=v_1^*(f(x_1))$ for all
$f\in C_0(X,E)$, then $v_1^*=v^*$ and $x_1=x$. Hence for each $y\in Y$, the point $x\in X$ satisfying (\ref{ext}) for some
$v^*\in {\rm ext}(E^*_1)$ is uniquely determined.
For any $f\in C_0(X,E)$ with  $Tf(y)=v_0$ and $\|Tf\|_\infty=1$ we have
$v^*(f(x))=1$ which implies that $\|f(x)\|=1=\|v^*\|$. This shows
that the set $S_0=\{e\in S(E) : v^*(e)=1\}$ is nonempty. On
the other hand, for any $f\in C_0(X,E)$ with $\|f\|_\infty=1$ and
$f(x)\in S_0$, it follows from the  above equality that
$w^*(Tf(y))=1$. This implies that $\|Tf(y)\|=1$ and since
$w^*(Tf(y))=1$ we conclude that $Tf(y)\in R=\{v_0\}$, i.e.
$Tf(y)=v_0$.

We claim that $Tf(y)=0$ for each $f\in C_0(X,E)$ vanishing on a
neighborhood of $x$.  For this, assume that $f\in C_0(X,E)$ such
that $\|f\|_\infty=1$ and $f$ vanishes on a neighbourhood of $U$ of
$x$. Fixing $e\in S_0$, we can find  $g\in C_0(X,E)$ such that ${\rm
coz}(g)\subseteq U$, $\|g\|_\infty=1$ and $g(x)=e$. Then
$\|f+g\|_\infty=1$. We note that  $g(x)=e$ and $\|g\|_\infty=1$ and hence,
by the above argument, we have $Tg(y)=v_0$. Similarly since
$f(x)+g(x)=e$ and $\|f+g\|_\infty=1$ we have $T(f+g)(y)=v_0$. This
concludes, by the real-linearity of $T$, that  $Tf(y)=0$. Since
each function $f\in C_0(X,E)$ with $f(x)=0$ is the  uniform limit of a
sequence $\{f_n\}$ in $C_0(X,E)$ such that each $f_n$ vanishes on a
neighborhood of $x$, we conclude that $Tf(y)=0$  for all $f\in
C_0(X,E)$  satisfying $f(x)=0$.

 Thus for each $y\in Y$ and the (unique) point $x\in X$ satisfying  (\ref{ext})
 \begin{quote}
 $f(x)=0$ implies that $Tf(y)=0$ for all $f\in C_0(X,E)$.
 \end{quote}
  Let $\Phi: Y\longrightarrow X$ be the map which associates to each $y\in Y$ the unique point $x\in X$ satisfying the above implication.

 For each $y\in Y$, we define  $V_y(e)= Tf_0(y)$, $e\in E$, where $f_0\in C_0(X,E)$ is such that $f_0(\Phi(y))=e$. We note that, by the above implication,  the definition of $V_y(e)$ is independent of  the function $f_0$ with the mentioned property. Then $V_y: E \longrightarrow F$ is a real-linear operator satisfying
 \begin{align} \label{eq}
 Tf(y) = V_y(f(\Phi(y)) \;\;\;\;\; f\in C_0(X,E).
 \end{align}
Since for each $y\in Y$ and $e\in S(E)$ we can choose a function $f_0\in C_0(X,E)$ with $\|f_0\|_\infty=1$ and $f_0(\Phi(y))=e$ it follows easily that $\|V_y\|\le 1$.

To prove that $\Phi$ is continuous, assume that $y_0\in Y$ and $U$ is an open neighbourhood of $\Phi(y_0)$ in $X$.
Put $M=\|V_{y_0}\|$. We note that surjectivity of $T$ and the description given in (\ref{eq}) shows that $V_{y_0}\neq 0$.  For a small enough $\epsilon>0$, choose $e\in S(E)$ with $\|V_{y_0}(e)\|>M-\epsilon$. Let $f\in C_0(X,E)$ such that $f(\Phi(y_0))=e$, $\|f\|_\infty=1$ and $\|f(x)\|<\frac{M-\epsilon}{2}$ for all $x\in X \backslash U$.
Then $W=\{y\in Y: \|Tf(y)\|> \frac{M-\epsilon}{2}\}$ is a neighbourhood of $y_0$ and the equality (\ref{eq}) implies that $\Phi(W) \subseteq U$. Hence $\Phi$ is continuous.

Finally, if $E$ satisfies the same condition as $F$, then using the same argument for $T^{-1}$ there are a continuous map $\Psi: X\longrightarrow Y$ and a family $\{W_x\}_{x\in X}$ of bounded real-linear operators from $F$ to $E$ such that $\|W_x\|\le 1$ and \[T^{-1}(g)(x)= W_x(g(\Psi(x))) \qquad (g\in C_0(Y,F)).\]
Then an easy verification shows that $\Psi=\Phi^{-1}$ and $V_y=W_{\Phi(y)}^{-1}$. In particular, $\Phi$ is a homeomorphism and each $V_y$ is a surjective isometry.

{\bf Case (ii)} Let $y\in Y$. Since $w_0^*\in {\rm ext}(F^*_1)$
there exists $v_0^*\in {\rm ext}(E^*_1)$ and a point $x_0\in X$
such that
\[w_0^*(Tf(y))= v_0^*(f(x_0))\qquad (f\in C_0(X,E)).\]
For each $w^* \in {\rm ext}(F^*_1)$ distinct from $w_0^*$, choose
$v^*\in {\rm ext}(E^*_1)$ and $x\in X$ such that
\[w^*(Tf(y))= v^*(f(x))\qquad (f\in C_0(X,E)).\]
We claim that $x=x_0$. Assume on the contrary that $x\neq x_0$.
Since $\|v_0^*\|=\|v^*\|=1$, for each $\epsilon>0$ we can find
$e_0,e\in S(E)$ such that $v_0^*(e_0)> 1-\epsilon$ and $v^*(e)>
1-\epsilon$. Choose $f\in C_0(X,E)$ with
$\|f\|_\infty=1$, $f(x_0)=e_0$ and $f(x)=e$. Then using the above
equalities we get
\[ w_0^*(Tf(y))>1-\epsilon, \;\; {\rm and} \;\; w^*(Tf(y))>1-\epsilon.\]
Thus
\[\|w_0^*+w^*\|\ge w_0^*(Tf(y))+w^*(Tf(y))>2-\epsilon,\]
that is $\|w_0^*+w^*\|\ge 2-\epsilon$ for all $\epsilon>0$. Hence
$\|w_0^*+w^*\|=2$ and consequently $w^*\in {\rm St}(w_0^*)=\{w_0^*\}$, a contradiction.
Hence $x=x_0$.

The argument above shows that for each $y\in Y$ there exists a
(unique) point $x\in X$ such that for each $w^*\in {\rm ext}(F^*_1)$ we
have
\[ w^*(Tf(y))=v^*(f(x))\qquad (f\in C_0(X,E))\]
for some $v^*\in {\rm ext}(E^*_1)$. Now for each $f\in C_0(X,E)$
with $f(x)=0$ we have $w^*(Tf(y))=0$ for all $w^*\in {\rm
ext}(F^*_1)$, which concludes that $Tf(y)=0$.
Then, as in Case (i), we can define a continuous map $\Phi: Y \longrightarrow X$ and a family $\{V_y\}_{y\in Y}$ of bounded real-linear maps from $E$ to $F$ satisfying the desired conditions.

{\bf  Case (iii).} For an arbitrary T-set $R$ in  $F$ we put
\[\Gamma_R=\{ w^*\in {\rm ext} (F^*_1): w^*(u)=\|u\|\ \ \   {\rm for \; all}\;  u\in R\}. \]
Then, using Krein Milman Theorem, $\Gamma_R$ is nonempty (see \cite[Lemma 7.2.4]{Flem2}). We note that  for any $w^*\in
\Gamma_R$ we have $R\subseteq \{u\in F : w^*(u)=\|u\|\}$ and since
for each $u_1, ..., u_n\in F$ with $w^*(u_i)=\|u_i\|$ we have
$\|u_1+\cdots +u_n\|=\|u_1\|+\cdots +\|u_n\|$, the maximality of
$R$ implies that
\[R=\{u\in F : w^*(u)=\|u\|\}.\]
Now for any $w^*\in  \Gamma_R$ and for any $y\in Y$ there are $v^*\in {\rm ext}(E^*_1)$ and $x\in X$ such that
 \begin{align} \label{1}
 w^*(Tf(y))=v^*(f(x))\;\;\;\;\; (f\in C_0(X,E)).
 \end{align}
 We claim that there is a T-set $S$ in  $E$ such that $v^*\in \Gamma_S$ and $T^{-1}(R,y)=(S,x)$. For this, choose an arbitrary function  $f_0\in C_0(X,E)$ with $\|f_0\|_\infty=1$ and $Tf_0\in (R,y)$. Then
 \[v^*(f_0(x))=w^*(Tf_0(y))=\|Tf_0(y)\|=1=\|v^*\|. \]
 Hence $\|f_0(x)\|=1= v^*(f_0(x))$, and consequently the set  $S=\{e\in E : v^*(e)=\|e\|\}$ is nonempty. Since for any $e_1,...,e_n\in S$ we have $\|e_1+\cdots +e_n\|=\|e_1\|+\cdots +\|e_n\|$ there exists a T-set $S_0$ in $E$ such that $S\subseteq S_0$.  We note that  $T^{-1}(R,y)=(S,x)$. Indeed,
  take  $f\in C_0(X,E)$ such that  $Tf\in (R,y)$. Then $\|Tf\|_\infty=\|Tf(y)\|$ and $Tf(y)\in R$. Hence
 \[\|f\|_\infty=\|Tf\|_\infty=\|Tf(y)\|=w^*(Tf(y))=v^*(f(x))\le \|f(x)\|\le \|f\|_\infty,\]
 which conclude that $\|f(x)\|=\|f\|_\infty$ and $v^*(f(x))=\|f(x)\|$, i.e. $f(x)\in S$. Hence   $f\in (S,x)$. Conversely, if $f\in (S,x)$, then  $\|f(x)\|=\|f\|_\infty$ and $f(x)\in S$. A similar argument shows that  $w^*(Tf(y))=\|Tf(y)\|=\|Tf\|_\infty$, that is  $Tf\in (R,y)$. This proves that $T^{-1}(R,y)=(S,x)$. Hence it suffices to show that $S=S_0$. Since $(R,y)$ is a T-set in $C_0(Y,F)$ and $T^{-1}$ is an isometry, it follows that $(S,x)$ is also a T-set in $C_0(X,E)$. The inclusion $(S,x) \subseteq (S_0,x)$ implies that $(S,x)=(S_0,x)$ and this easily implies that $S=S_0$, as desired.

 Since $F$ satisfies  property (D), using the same argument as in \cite[Theorem 7.2.13]{Flem2} we can show that for each $y\in Y$, the point $x\in X$ in equality (\ref{1}) is unique for all $w^*\in \cup  \Gamma_R $ where the union is taken over all T-sets $R$ in $F$.  Indeed, assume that $w^*\in \Gamma_R$ and $w_1^*\in \Gamma_{R_1}$ for some T-sets $R$ and $R_1$ in $F$ and
 \[ T^{-1}(R,y)=(S,x) \;\;{\rm and} \; \; T^{-1}(R_1,y)=(S_1,x_1)\]
 where $x, x_1\in X$ are distinct and $S$ and $S_1$ are T-sets in $E$.  By assumption we have either $R\cap R_1=\{0\}$ or there exists a T-set $L$ in $F$ such that
 $R\cap L= R_1\cap L=\{0\}$. In the first case we have $(R,y)\cap (R_1,y)=\{0\}$ and consequently $(S,x)\cap (S_1,x_1)=\{0\}$. Since $x\neq x_1$ we can find functions $f,g\in C_0(X)$ such that $0\le f,g \le 1$, $fg=0$, $f(x)=1$ and $g(x_1)=1$. Then for nonzero elements $e\in S$ and $e_1\in S_1$ the nonzero function $h=fe+ge_1$ in $C_0(X,E)$ belongs to the intersection $(S,x)\cap(S_1,x_1)$, which is impossible. Similarly, the other case implies that $x=x_1$.

Hence  we can define a map $\Phi: Y \longrightarrow X$ such that for a given $y\in Y$ and each  $w^*\in \cup_{R}\Gamma_R$
there exists $v^*\in \cup_{S} \Gamma_S$ satisfying
\[w^*(Tf(y))=v^*(f(\Phi(y)))\;\;\;\;\; (f\in C_0(X,E)).\]
 Now if $f\in C_0(X,E)$  such that $f(\Phi(y))=0$, then by the above equality,
  $w^*(Tf(y))=0$ for all  $w^*\in \cup_{R}\Gamma_R$. Choose a  T-set $R_0$ in $F$ containing $Tf(y)$ and let $w^* \in \Gamma_{R_0}$. Then  $0=w^*(Tf(y))=\|Tf(y)\|$, which implies that $Tf(y)=0$. Hence $Tf(y)=0$ for  all $f\in C_0(X,E)$ with $f(\Phi(y))=0$. Now, as in Case (i), $\Phi$ is a continuous
  map and we can find  a family $\{V_y\}_{y\in Y}$ of real-linear operators from $E$ to $F$ satisfying the desired properties.
  \end{proof}

\vspace*{.2cm} We define the following property for an isometry
$T: A \longrightarrow B$ between subspaces $A$ and $B$ of $C(X,E)$
and $C(Y,F)$, respectively and then we compare it with the similar
properties considered in \cite{Ar}, \cite{Hos}, and \cite{Bot,Ran}.

\begin{defn}  Let $X$ and $Y$ be compact Hausdorff spaces and let $E$ and $F$ be real or complex normed spaces. Assume that  $A$ and $B$ are  subspaces of $C(X,E)$ and $C(Y,F)$, respectively, equipped with the norms of the form
\[ \|\cdot \|_A= \max(\|\cdot\|_\infty, p(\cdot)) \;{\rm and}\; \|\cdot \|_B= \max(\|\cdot\|_\infty, q(\cdot)),\]
where $p$ and $q$ are seminorms on $A$ and $B$, respectively, whose kernels contain the constant functions. We say that a surjective real-linear isometry  $T:A \longrightarrow B$ has property  {\rm (St)} if
\begin{quote}
{\rm (St)} For each $u\in F$ and $y_0\in Y$ there exists $v\in
S(E)$ such that $\|T\hat{v}(y_0)+u\|>\|u\|$, i.e.
$\frac{T(\hat{v})(y_0)}{\|T(\hat{v})(y_0)\|} \in \St_w(u)$.
\end{quote}
\end{defn}
\vspace*{.2cm}

\noindent
{\bf Remark.} The following observations are simple.

(i) Property {\bf Q} in \cite{Bot,Ran} implies both  (St) and property {\bf P} in \cite{Ar}.

(ii)  In the case that $F$ is strictly convex (this is assumed in
\cite{Ar} and \cite{Hos}), property {\bf P}  implies (St).

(iii) If $X$ is a compact subset of $\Bbb R$, $E,F$ are  strictly
convex and $A=AC(X,E)$ and $B=AC(Y,F)$, then, by Lemmas 3.6 and 3.7  in \cite{Hos},  the condition $a,b\in \mathcal{N}(T)$ in \cite{Hos} implies that $T$ and $T^{-1}$ satisfy {\bf P} and consequently satisfy (St). Here $a= \min Y$ and $b=\max Y$  and $\mathcal{N}(T)= \{y\in Y: T\hat{e}(y)\neq 0\;{\rm for\;\; some \;\;} e\in E \}$.

 \begin{prop}\label{Pr}
Let $X$ and $Y$ be compact Hausdorff spaces and let $E$ and $F$ be real or complex  normed spaces, not assumed to be complete. Assume that
 $A$ and $B$ are subspaces of $C(X,E)$ and $C(Y,F)$, respectively containing constants and  $\|\cdot\|_A$ and $\|\cdot \|_B$ are norms on $A$ and $B$ such that
 \[ \|\cdot \|_A= \max(\|\cdot\|_\infty, p(\cdot)) \;{\rm and}\; \|\cdot \|_B= \max(\|\cdot\|_\infty, q(\cdot))\]
for some seminorms $p$ and $q$ on $A$ and $B$, respectively, whose kernels contain the constants.
If  $T:A \longrightarrow B$ is a surjective real-linear isometry and $T$ and $T^{-1}$ satisfy {\rm (St)}, then  $T$ is an isometry with respect to the supremum norms on $A$ and $B$.
  \end{prop}
 \begin{proof}
Using property (St) a minor modification of the proofs of \cite[Lemma 1.2]{Ran} shows that if $f\in A(X,E)$
satisfies $q(Tf)< \|Tf\|_\infty$, then $p(f)\le \|f\|_\infty$. Then the same argument as in \cite[Proposition 1.3]{Ran} implies that $\|Tf\|_\infty \le \|f\|_\infty$ for all $f\in A(X,E)$. Similarly, the other inequality holds, i.e. $T$ is an isometry with respect to the supremum norms.
 \end{proof}
Next theorem is our main result.

\begin{theorem}\label{main2}
Let $X$ and $Y$ be compact Hausdorff spaces and let $E$ and $F$ be
real or complex  normed spaces, not assumed to be complete. Assume
that $A$ and $B$ are dense subspaces of $C(X,E)$ and $C(Y,F)$,
respectively containing constants and  $\|\cdot\|_A$ and $\|\cdot
\|_B$ are norms on $A$ and $B$ such that
 \[ \|\cdot \|_A= \max(\|\cdot\|_\infty, p(\cdot)) \;{\rm and}\; \|\cdot \|_B= \max(\|\cdot\|_\infty, q(\cdot))\]
for some seminorms $p$ and $q$ on $A$ and $B$, respectively, whose
kernels contain the constants. If $F$ satisfies one of the
following conditions

{\rm (i)} $S(F)$ contains an element $v_0$ with ${\rm
St}(v_0)=\{v_0\}$,

{\rm (ii)} ${\rm ext}(F_1^*)$ contains an element $w_0^*$ with
${\rm St}(w_0^*)=\{w_0^*\}$,

{\rm(iii)} $F$ satisfies property {\rm (D)},

\noindent then for any surjective real-linear isometry $T:(A, \|\cdot \|_A) \longrightarrow (B, \|\cdot\|_B)$
such that $T$ and $T^{-1}$ satisfy {\rm (St)}, there exist a
continuous map $\Phi:Y \longrightarrow X$, a family $\{V_y\}_{y\in
Y}$ of bounded real-linear operators from $E$ to $F$ with $\|V_y\|\le 1$ such that
\[Tf(y)= V_y(f(\Phi(y))) \qquad (f\in A).\]
Moreover, if $E$ satisfies the same condition as $F$, then $\Phi$ is a homeomorphism and each $V_y$ is a surjective  isometry.
\end{theorem}
\begin{proof}
This is immediate from Proposition \ref{Pr} and Theorem  \ref{main1}.
 \end{proof}
As it was mentioned before, all vector-valued function spaces in Example \ref{exam} satisfy the requirement of the above theorem. Hence this theorem characterizes surjective isometries between these function spaces with respect to the norms considered on them.

 \end{document}